\documentclass[10pt,a4paper]{article} 

\usepackage[utf8]{inputenc}
\usepackage[T1]{fontenc}
\usepackage[english]{babel} 
\usepackage{mathpazo} 
\usepackage{microtype} 
\usepackage{authblk}
\usepackage{xcolor}

\usepackage[a4paper, left=2.5cm, right=2.5cm, top=2.5cm, bottom=2.5cm]{geometry}
\usepackage{fancyhdr}
\usepackage{titlesec}

\titleformat{\section}{\large\bfseries}{\thesection}{1em}{}
\titleformat{\subsection}{\normalsize\bfseries}{\thesubsection}{1em}{}

\usepackage{amsmath,amssymb,amsthm}
\usepackage{mathtools}
\usepackage{mathrsfs}

\usepackage{graphicx}
\usepackage{float}
\usepackage{booktabs}
\usepackage[font=small,labelfont=bf]{caption}

\usepackage{enumitem}
\setlist[itemize]{noitemsep, topsep=0pt}

\usepackage{csquotes}
\usepackage[numbers,sort&compress]{natbib}
\usepackage{hyperref}
\hypersetup{
    colorlinks=true,
    linkcolor=blue,
    filecolor=magenta,
    urlcolor=cyan,
    citecolor=red,
    pdftitle={Stochastic epidemiological modeling},
    pdfauthor={L.A. Valencia, A. Moran, D. Cataño}
}

\numberwithin{equation}{section}

\theoremstyle{plain}
\newtheorem{theorem}{Theorem}[section]
\newtheorem{lemma}[theorem]{Lemma}
\newtheorem{proposition}[theorem]{Proposition}
\newtheorem{corollary}[theorem]{Corollary}

\theoremstyle{definition}

\theoremstyle{remark}

\begin{document}

\title{\textbf{Modeling Transmission Intensity in SI Epidemics via CIR and Jacobi Processes: Asymptotic Results and Preliminary Intervention Strategies}}

\author[1,]{Leon A. Valencia}
\author[2]{Ra\'ul Alejandro Mor\'an-V\'asquez}
\author[3]{Duvan H. Cata\~no Salazar}
\affil[1,2,3]{Instituto de Matem\'aticas, Universidad de Antioquia, Calle 67 No. 53-108, Medell\'in 050010, Colombia.}


\maketitle

\noindent 
\begin{abstract}
This paper introduces a way of modeling the epidemic transmission rate using a stochastic process of the form $(\beta_t = \varphi(t)P_t : t \ge 0)$, where the positive deterministic function $\varphi(t)$ models the impact of a public health intervention and $P_t$ describes the stochastic evolution of the infection rate in the absence of any control measures. We establish general asymptotic results for an SI model governed by $(\beta_t : t \ge 0)$, showing that the asymptotic behavior is determined by the integrated intensity process $(H_t =\int_0^t \beta_s \, ds : t \ge 0)$. We study the intrinsically bounded Jacobi process and the Cox--Ingersoll--Ross (CIR) process as models for $(P_t : t \ge 0)$; both exhibit almost surely positive sample paths. We highlight that in the case of non-intervention $(\varphi \equiv 1)$, the process $(H_t : t \ge 0)$ is considerably more analytically tractable. Finally, we present numerical simulations for both models in two different scenarios: the case of non-intervention $(\varphi(t)=1)$ and the case of a successful intervention strategy (where $\int_0^\infty \varphi(t) \, dt < \infty$) modeled using exponential decay $\varphi(t) = e^{-\alpha t}$ for both models.
\end{abstract}
\vspace{0.5cm}
\vspace{0.5cm}

\section{Introduction}

The quantitative study of infectious diseases fundamentally relies on describing how susceptible and infected individuals meet and interact. The first study of this type of phenomenon is generally attributed to \cite{Hamer1906}. In that work, it was proposed that the rate of new infections is proportional to the product of the densities of both groups. This assumption still remains in many modern epidemiological models. The conceptual approach introduced in \cite{Ross1911} was formalized with greater mathematical rigor in \cite{Kermack1927}. Deterministic modeling continues to serve as a reference to understand this phenomenon \cite{Murray2002, Hethcote2000, Brauer2012}.

Below we provide a brief description of the model. Consider a population of $N$ individuals divided into two groups, susceptible ($S$) and infected ($I$), whose infection dynamics follow the following deterministic system:
\begin{equation*}
\left\{
\begin{aligned}
    \frac{dS(t)}{dt} &= -\beta \frac{S(t)I(t)}{N}, \\
    \frac{dI(t)}{dt} &= \beta \frac{S(t)I(t)}{N},
\end{aligned}
\right.
\end{equation*}
for $t \ge 0$, where $\beta > 0$ represents the transmission rate and the total population $S(t) + I(t) = N$ remains constant. Here, $\beta$ is traditionally defined as the product of the average number of contacts per unit of time and the probability of transmission per contact. 

By assuming a normalized population $N=1$ and noting that $I(t) = 1 - S(t)$, the system reduces to a single ordinary differential equation (ODE) for the density of susceptibles:
\begin{equation}\label{eq1}
\frac{dS_t}{dt} = -\beta S_t(1 - S_t), \quad S_0 = s_0 \in (0,1).
\end{equation}
While Equation \eqref{eq1} describes a purely deterministic evolution, real-world transmission rates are subject to environmental fluctuations and unpredictable human behavior. To capture this uncertainty, we transition to a stochastic framework by replacing the constant parameter $\beta$ with a stochastic process $(\beta_t : t \ge 0)$, leading to the stochastic differential equation (SDE) studied in this work.

The literature offers several methods to incorporate the randomness that naturally occurs in these systems. For instance, \cite{Kegan2005} investigated the effects of treating the initial condition as a random variable rather than a fixed value. While that study analyzed the random time required to reach a specific infection threshold, randomizing the initial state does not alter the system's asymptotic behavior, entailing that the entire population becomes infected. This suggests that to model phenomena where a significant portion of the population remains susceptible, randomness must be incorporated into the evolution of the process, rather than solely in its initialization. 

A common approach involves transforming the deterministic equation \eqref{eq1} into a Stochastic Differential Equation (SDE) of the form:
\begin{equation}\label{eq2}
    dS_t = \mu(S_t)dt + G(S_t)dW_t, \quad S_0 = s_0 \in (0,1),
\end{equation}
where $W_t$ is a standard Wiener process defined on a complete filtered probability space $(\Omega, \mathcal{F}, \{\mathcal{F}_t\}_{t \ge 0}, \mathbb{P})$. The drift coefficient $\mu(\cdot)$ and the diffusion coefficient $G(\cdot)$ characterize the deterministic trend and the intensity of the stochastic fluctuations, respectively. Typically, noise is introduced by perturbing the transmission rate as $\beta_t dt = \beta dt + \sigma dW_t$, where $\sigma > 0$ scales the intensity of the white noise. 

The mathematical regularity of this standard framework is well-established. Notably, by employing Lyapunov functions, it is possible to show that trajectories stay within the interval $(0,1)$ almost surely \cite{Khasminskii2012}. Furthermore, if the noise intensity $\sigma$ is sufficiently small, the stochastic process $(S_t)_{t\ge 0}$ possesses a unique limiting distribution with support on $(0,1)$, preventing the total extinction of the susceptible population.

In this paper, we explore a modeling approach that differs from those traditionally well-studied in the mathematical literature. Instead of perturbing the rate with additive noise, we treat the infection rate as a stochastic process that is inherently positive by construction. Specifically, we utilize the Cox--Ingersoll--Ross (CIR) model and a Jacobi-type process on $[0, a]$ to modulate the transmission dynamics. While these processes are mathematically well-defined and preserve the positive nature of the transmission rate, they share a common asymptotic drawback: without further measures, the population eventually tends toward a fully infected state. 

This motivates the implementation and analysis of an ``intervention strategy'' within this stochastic framework. Formally, we define the transmission rate as a modulated process $\beta_t = \varphi(t) P_t$, where $P_t$ represents the intrinsic stochasticity of the infection (modeled by CIR or Jacobi) and $\varphi(t)$ is a deterministic function representing public health actions. This multiplicative structure allows the intervention to scale the magnitude of the stochastic fluctuations, reflecting that a reduction in the average contact rate also reduces the absolute variance of the infection process. Under certain integrability conditions on $\varphi$, we show that this strategy can shift the asymptotic regime, allowing for the long-term survival of a susceptible fraction of the population.

The rest of this article is organized as follows. Section \ref{sec:general_model} establishes the general stochastic framework for the SI model and provides proofs for the system behavior under different intervention scenarios. Sections \ref{sec:jacobi} and \ref{sec:cir} introduce the two candidates for the stochastic intensity: the bounded Jacobi process and the unbounded CIR process. In Section \ref{sec:random_time}, the focus shifts to risk assessment; here, we analyze the random time needed to reach critical thresholds specifically within the CIR framework under a non-intervention regime, deriving an analytical ``risk monitor'' based on the Chernoff bound. Finally, Section \ref{sec:numerico} provides a numerical comparison to demonstrate how the choice of stochastic driver affects the estimation of extreme risks, followed by concluding remarks in Section \ref{sec:conclusion}.

\section{General Model Formulation and Asymptotic Results}
\label{sec:general_model}

To build our framework, we consider an SI model \textcolor{red}{in} a population of unit size ($N=1$), where $S_t$ represents the proportion of individuals still susceptible at time $t$. Unlike traditional deterministic approaches, we assume the transmission rate is driven by a stochastic process $(\beta_t)_{t\geq 0}$ with continuous and almost surely (a.s.) positive sample paths. Under these assumptions, the epidemic dynamics are described by the following equation:

\begin{equation}\label{eq:si-random}
    \frac{dS_t}{dt} = - \beta_t\, S_t \bigl(1 - S_t \bigr), \qquad S_0 = s_0 \in (0,1).
\end{equation}

A key advantage of this formulation is its analytical tractability, enabling an explicit solution to the system for any given noise realization. The following proposition provides the exact epidemic trajectory.

\begin{proposition}\label{prop:exact-solution}
Let $(\beta_t)_{t \ge 0}$ be a stochastic process with a.s. continuous sample paths. Equation \eqref{eq:si-random} admits a unique global solution given by
\begin{equation}\label{eq:si-solution}
    S_t = \left[ 1 + \left(\frac{1}{s_0} - 1\right)\exp\left(\int_0^t \beta_r\, dr\right) \right]^{-1}, \qquad t \ge 0.
\end{equation}
\end{proposition}

\textit{Remark.} The solution \eqref{eq:si-solution} is obtained pathwise using classical methods. This expression indicates that the state of the epidemic is determined by the stochastic process $(H_t = \int_0^t \beta_s \, ds : t \ge 0)$, which we will call the cumulative transmission intensity.

In our approach, we decompose the transmission rate as $\beta_t = \varphi(t) P_t$. Here, $P_t$ models the randomness of the transmission rate in the absence of any interventions, while $\varphi(t)$ is a deterministic function that models the impact of public health interventions. This decomposition allows us to analyze the final state of the epidemic ($S_\infty = \lim_{t \to \infty} S_t$) by studying the asymptotic behavior of $H_t$. Specifically, a portion of the population escapes infection (i.e., $S_\infty > 0$) if and only if $H_t$ remains bounded as $t \to \infty$ almost surely.


The following theorem formalizes the conditions under which an intervention successfully prevents total infection.

\begin{theorem}[Asymptotic State under Successful Intervention]
\label{intervention}
Let $\beta_t = P_t \varphi(t)$, where $\varphi$ is a successful intervention (i.e., $\int_0^\infty \varphi(s) \, ds < \infty$) and $(P_t)_{t\ge 0}$ is a stochastic process with a.s. continuous positive paths and a bounded mean $\sup_{t \ge 0} \mathbb{E}[P_t] \le M < \infty$. Then, $S_\infty \in (0, s_0)$ a.s.
\end{theorem}

\begin{proof}
The non-decreasing nature of $H_t$ ensures that $H_\infty$ is a well-defined random variable (taking values in $[0, \infty]$). To guarantee that $S_\infty$ does not vanish, it is sufficient to show that $H_\infty$ is finite almost surely. By applying Fubini's Theorem to the expected value of $H_\infty$, we obtain:
\begin{align*}
    \mathbb{E}[H_\infty] = \mathbb{E} \left[ \int_0^\infty P_s \varphi(s) \, ds \right] = \int_0^\infty \mathbb{E}[P_s] \varphi(s) \, ds \le M \int_0^\infty \varphi(s) \, ds < \infty.
\end{align*}
Since the expectation is finite, it follows that $H_\infty < \infty$ a.s. Consequently, from the explicit solution given in \eqref{eq:si-solution}, the susceptible fraction $S_\infty$ remains strictly positive, and specifically $S_\infty \in (0, s_0)$ due to the positivity of $H_\infty$ and the initial condition.
\end{proof}

Conversely, without control measures, the system returns to its natural trajectory, even within a stochastic environment.

\begin{theorem}[Asymptotic State without Intervention, $\varphi \equiv 1$]
\label{no-intervention}
Suppose $(P_t)_{t\ge 0}$ is an ergodic Markov process with a unique stationary distribution $\pi$ supported on $(0, \infty)$ and a finite mean $\int_0^\infty x \pi(x) \, dx > 0$. Then, $H_\infty = \infty$ a.s., which implies $S_\infty = 0$ a.s.
\end{theorem}

\begin{proof}
The result follows from the Ergodic Theorem for Markov processes. Given that $\varphi \equiv 1$, the integrated intensity $H_t$ is simply $\int_0^t P_s \, ds$. The time average satisfies:
\begin{equation*}
    \lim_{t \to \infty} \frac{H_t}{t} = \lim_{t \to \infty} \frac{1}{t} \int_0^t P_s \, ds = \int_0^\infty x \pi(x) \, dx > 0 \quad \text{a.s.}
\end{equation*}
This linear growth implies that $H_t \approx t \cdot K$ for large $t$ and some constant $K>0$. Consequently, the integrated intensity diverges ($H_\infty = \infty$) almost surely, resulting in the extinction of the susceptible population according to the explicit solution \eqref{eq:si-solution}.
\end{proof}

\subsection{Model I: A Bounded Rate via the Jacobi Process}
\label{sec:jacobi}

We first consider a scenario where the transmission intensity $(P_t^J)_{t\geq 0}$ is governed by a Jacobi process. This formulation is biologically sound, satisfying the physical requirement that the transmission rate remains positive and bounded by the constraints of human interaction. The coupled dynamics are described by the following system of SDEs:
\begin{equation}\label{JacobiModel}
    \begin{cases}
        dS_t = -P_t^J \varphi(t) S_t(1 - S_t) \, dt, \\
        dP_t^J = \theta(\mu - P_t^J) \, dt + \sigma \sqrt{P_t^J(a - P_t^J)} \, dW_t,
    \end{cases}
\end{equation}
with initial conditions $S_0 = s_0 \in (0,1)$ and $P_0^J = p_0 \in (0,a)$. Here, $\theta > 0$ represents the mean-reversion speed, $\mu \in (0,a)$ is the long-term mean, and $\sigma > 0$ denotes the volatility coefficient (see \cite{delbaen2002interest}). 

The parameter $a > 0$ defines the upper bound of the transmission rate, representing the maximum possible intensity of infection under the specific social and environmental conditions of the population. While the CIR process (discussed below) allows for arbitrarily large spikes in transmission, the Jacobi process restricts $P_t^J$ to the compact interval $[0, a]$. This boundedness is consistent with the reality that even in the absence of interventions, the number of effective contacts per unit of time is limited by physical and temporal constraints. For the process to remain strictly within the boundaries $(0, a)$ almost surely, we assume that the parameters satisfy the condition $\theta \mu > 0$ and $\theta(a - \mu) > 0$, ensuring that the drift is strong enough to push the process away from the absorbing states at $0$ and $a$.

Next, we state the properties of the Jacobi process required to ensure the model is well-defined. Provided the parameters are positive, the process is confined to the interval $(0,a)$ for all $t \ge 0$ almost surely. Additionally, if the parameters satisfy the Feller-like conditions $2\theta\mu \ge \sigma^2 a$ and $2\theta(a-\mu) \ge \sigma^2 a$, the process admits a unique stationary Beta distribution with support on $(0,a)$. In this case, the first moment of the process is given by $\mathbb{E}[P_t^J] = \mu + (p_0 - \mu)e^{-\theta t}$ for $t \ge 0$. 

Following the framework established in Section \ref{sec:general_model}, the asymptotic behavior of the epidemic depends on the convergence of the integrated intensity process $H_t^J$. The following lemma provides an explicit expression for the expected cumulative pressure, which is used to verify the stability conditions of Theorem \ref{intervention}.

\begin{lemma}\label{lem:mean-jacobi}
Under the Jacobi framework, the expected Integrated Transmission Intensity at time $t \geq 0$ is given by
\begin{equation*}
    \mathbb{E}[H_t^J] = \int_0^t \left[ \mu + (p_0 - \mu)e^{-\theta s} \right] \varphi(s) \, ds.
\end{equation*}
\end{lemma}

\begin{proof}
The result follows by applying Fubini's Theorem to $H_t^J = \int_0^t P_s^J \varphi(s) \, ds$ and substituting the explicit expression for the first moment of $P_s^J$.
\end{proof}

\subsection{Model II: A Non-Negative Rate via the CIR Process}
\label{sec:cir}

As an alternative to the previous framework, we consider modeling $(P_t^C : t \ge 0)$ as a Cox--Ingersoll--Ross (CIR) process. While this process has non-negative sample paths, it lacks an upper bound, which constitutes a significant difference from the Jacobi process. The resulting joint dynamics are governed by the following system:
\begin{equation}\label{CIRModel}
    \begin{cases} 
        dS_t = -P^C_t \varphi(t) S_t(1 - S_t) \, dt, \\ 
        dP^C_t = \kappa(\eta - P^C_t) \, dt + \sigma \sqrt{P^C_t} \, dW_t, 
    \end{cases}
\end{equation}
where $\kappa, \eta$, and $\sigma$ are positive constants. Although the process maintains non-negative paths for any positive parameters, the Feller condition $2\kappa \eta > \sigma^2$ ensures that the paths remain strictly positive almost surely \cite{cox1985theory}. 

Under this condition, the process admits a unique stationary Gamma distribution with shape parameter $2\kappa\eta/\sigma^2$ and scale parameter $\sigma^2/2\kappa$. The expected value of the process, starting from an initial intensity $P_0^C = p_0 > 0$, is given by $\mathbb{E}[P^C_t] = \eta + (p_0 - \eta)e^{-\kappa t}$, which converges to the long-term mean $\eta$ as $t \to \infty$.

The following lemma provides the analytical form of the expected integrated intensity under the CIR model. This result is used to verify the integrability condition $\mathbb{E}[H_\infty] < \infty$ required for the long-term survival of the susceptible population.

\begin{lemma}
Under the CIR framework, the expected Integrated Transmission Intensity at time $t \geq 0$ is given by
\begin{equation*}
    \mathbb{E}[H_t^{CIR}] = \int_0^t \left[ \eta + (p_0 - \eta)e^{-\kappa s} \right] \varphi(s) \, ds.
\end{equation*}
\end{lemma}

A key advantage of the CIR process lies in its affine diffusion structure. By the Feynman--Kac Theorem, this property yields a closed-form expression for the Laplace transform of $H_t^{CIR}$ under the non-intervention regime ($\varphi \equiv 1$).
\begin{proposition}\label{prop:laplace}
For any $\lambda \ge 0$, the Laplace transform of $H_T^{CIR}$ is given by
\begin{equation}
\mathbb{E}[e^{-\lambda H_T^{CIR}} \mid P_0 = p_0] = \exp\{-A(T,\lambda) - B(T,\lambda)p_0\},
\end{equation}
where $A(T,\lambda)$ and $B(T,\lambda)$ are determined by a system of ordinary Riccati differential equations.
\end{proposition}

\begin{proof}[Proof Sketch]
The derivation relies on the connection between SDEs and partial differential equations. By the Feynman--Kac formula, the function $u(t, x) = \mathbb{E}[e^{-\lambda \int_0^t P_s \, ds} \mid P_0 = x]$ satisfies the following PDE involving the infinitesimal generator $\mathcal{L}$ of the process:
\begin{equation}\label{eq:PDE_FeynmanKac}
\frac{\partial u}{\partial t} = \mathcal{L}u - \lambda x u, \quad u(0, x) = 1.
\end{equation}
For the CIR process, the operator $\mathcal{L}$ is defined as 
\begin{equation*}
\mathcal{L}u = \kappa(\eta - x)\frac{\partial u}{\partial x} + \frac{1}{2}\sigma^2 x \frac{\partial^2 u}{\partial x^2}.
\end{equation*} 
By assuming an exponential-affine structure $u(t, x) = \exp(-A(t) - B(t)x)$, the PDE \eqref{eq:PDE_FeynmanKac} reduces to a system of Riccati-type ODEs:
\begin{align*}
    B'(t) &= \lambda - \kappa B(t) - \frac{1}{2}\sigma^2 B^2(t), \quad B(0)=0, \\
    A'(t) &= \kappa \eta B(t), \quad A(0)=0.
\end{align*}
The resolution of such systems is well-established in the literature \cite{Lamberton2007, Glasserman2003}, yielding the explicit solutions:
\begin{align*}
    B(t, \lambda) &= \frac{2\lambda(e^{\gamma t} - 1)}{(\gamma + \kappa)(e^{\gamma t} - 1) + 2\gamma}, \\
    A(t, \lambda) &= \frac{2\kappa\eta}{\sigma^2} \ln \left( \frac{2\gamma e^{(\gamma+\kappa)t/2}}{(\gamma + \kappa)(e^{\gamma t} - 1) + 2\gamma} \right),
\end{align*}
where $\gamma = \sqrt{\kappa^2 + 2\sigma^2 \lambda}$. These expressions provide the closed-form Laplace transform.
\end{proof}
\subsection{{Random Time to Reach Infection Thresholds: The Non-Intervention Case}}
\label{sec:random_time}

In \cite{Kegan2005}, the authors investigate the time required for the proportion of susceptible individuals to reach a specific threshold from a random initial state. Motivated by this, our primary interest lies in characterizing the distribution of this random hitting time. While a closed-form distribution remains elusive for the unmitigated CIR model, we establish a rigorous bound for its cumulative distribution function. Let $S_0 = s_0 \in (0,1)$ be the initial susceptible fraction, and let $x \in (0, s_0)$ be a target threshold. Let $\tau \equiv \tau(s_0, x)$ denote the hitting time, defined as the random time at which the $S_\tau = x$.

From \eqref{eq:si-solution}, the event $\{\tau \leq t\}$ is equivalent to the integrated intensity $H_t^{CIR}$ exceeding a critical level:
$$\mathbb{P}(\tau \leq t) = \mathbb{P}(H_t^{CIR} \ge M),$$
where $M(s_0, x) = \ln\left[\frac{s_0(1-x)}{x(1-s_0)}\right]$ represents the cumulative infection pressure required to reach the target prevalence. We can bound this probability using Chernoff's inequality:

\begin{equation}
\mathbb{P}(\tau \leq t) \leq \inf_{\lambda > 0} e^{-\lambda M} \mathbb{E}[e^{\lambda H_t^{CIR}}].\label{Chernoff}    
\end{equation}
In the baseline scenario ($\varphi \equiv 1$), the moment-generating function (MGF) of $H_t^{CIR}$ is analytically tractable due to the affine structure of the process, allowing for a precise characterization of this hitting time distribution.

A similar argument allows us to characterize the Moment Generating Function (MGF) of the integrated intensity. Unlike the Laplace transform, the MGF is only well-defined for values of the parameter where the expectation does not diverge, leading to a restricted domain.

\begin{proposition}\label{prop:mgf_cir}
Let $H_t^{CIR}$ be the integrated intensity of a CIR process. For any $\lambda \in (0, \lambda_c)$, where $\lambda_c = \frac{\kappa^2}{2\sigma^2}$, the MGF is given by:
\begin{equation*}
\mathbb{E}\left[e^{\lambda H^{CIR}_t} \mid P_0 = p_0\right] = \exp\bigl\{A(t,\lambda) + B(t,\lambda)p_0\bigr\},
\end{equation*}
where the functions $A$ and $B$ are as defined in Proposition \ref{prop:laplace}, substituting $\lambda$ with $-\lambda$ and setting $\gamma = \sqrt{\kappa^2 - 2\sigma^2\lambda}$.
\end{proposition}

\begin{proof}
The proof follows the same steps as in the Laplace transform case by applying the Feynman--Kac Theorem to the expectation $u(t, x) = \mathbb{E}[e^{\lambda \int_0^t P_s \, ds} \mid P_0 = x]$. The resulting system of Riccati equations:
\begin{align*}
    B'(t) &= \frac{1}{2}\sigma^2 B^2(t) - \kappa B(t) + \lambda, \quad B(0)=0, \\
    A'(t) &= \kappa\eta B(t), \quad A(0)=0,
\end{align*}
possesses a real solution if and only if the discriminant $\Delta = \kappa^2 - 2\sigma^2\lambda$ is positive. This condition is satisfied for $\lambda < \lambda_c$, ensuring that the integral intensity does not explode and the MGF remains finite.
\end{proof}
The following corollary is an immediate consequence of \eqref{Chernoff} and Proposition \ref{prop:mgf_cir}.

\begin{corollary}
The probability of reaching the infection threshold before time $t$ is bounded by:
\begin{equation}\label{eq:optimization_problem}
    \mathbb{P}(\tau \leq t) \leq \inf_{\lambda \in (0, \lambda_c)} \exp\bigl( -\lambda M + A(t,\lambda) + B(t,\lambda)p_0 \bigr).
\end{equation}
\end{corollary}

Next, we aim to further explore the inequality obtained in the corollary above. Let $t > 0$ be fixed, and let $f:(0, \lambda_c) \to \mathbb{R}$ be defined by 
\begin{equation}
    f(\lambda) = \exp(-\lambda M + \Lambda_t(\lambda))
\end{equation}
where $\Lambda_t(\lambda) = A(t, \lambda) + B(t, \lambda)p_0 = \ln\left(\mathbb{E}_{p_0}\left[e^{\lambda H^{CIR}_t}\right]\right)$. Here, we define $\mathbb{E}_{p_0}[\cdot] = \mathbb{E}[\cdot \mid P_0 = p_0]$ as the expectation conditioned on the initial condition $p_0$. We then have the following proposition:

\begin{proposition}
\label{prop:asymptotic_risk}
Let $t>0$ be fixed. If $M > \mathbb{E}_{p_0}[H_t^{CIR}]$, then:

\begin{enumerate}
    \item[(i)] $f$ is strictly convex on $(0, \lambda_c)$ and reaches its minimum value in $(0, \lambda_c)$. 
    \item[(ii)] Let $\lambda^*(M)$ be the point that minimizes the function $f(\lambda)$ in $(0, \lambda_c)$. Then, $\lambda^*(M)$ is a strictly increasing function of $M$, and
    $$
    \lim_{M\uparrow\infty}\lambda^*(M) = \lambda_c = \frac{\kappa^2}{2\sigma^2}.
    $$
\end{enumerate}
\end{proposition}

\begin{proof}
(i) To show the strict convexity of $f(\lambda)$, it is sufficient to analyze the function $G(\lambda) = -\lambda M + \Lambda_t(\lambda)$ for $\lambda \in (0, \lambda_c)$. First, we analyze the differentiability of $\Lambda_t(\lambda)$. Let $\lambda \in (0, \lambda_c)$ be fixed, and let $\epsilon > 0$ be such that $(\lambda - \epsilon, \lambda + \epsilon) \subset (0, \lambda_c)$. For any $h \in (-\epsilon, \epsilon)$ and $k \in \mathbb{N}$, note that $x^k e^{(\lambda+h)x} \leq C e^{(\lambda+\epsilon)x}$ for all $x > 0$, where $C$ is a constant depending on $k$. This allows the application of the Dominated Convergence Theorem to obtain:

$$
\Lambda_t^{\prime}(\lambda) = \frac{\mathbb{E}_{p_0}[H_t^{CIR}e^{\lambda H_t^{CIR}}]}{\mathbb{E}_{p_0}[e^{\lambda H_t^{CIR}}]},
$$
and
$$
\Lambda_t^{\prime\prime}(\lambda) = \frac{\mathbb{E}_{p_0}[e^{\lambda H_t^{CIR}}]\mathbb{E}_{p_0}[(H_t^{CIR})^2 e^{\lambda H_t^{CIR}}] - (\mathbb{E}_{p_0}[H_t^{CIR}e^{\lambda H_t^{CIR}}])^2}{(\mathbb{E}_{p_0}[e^{\lambda H_t^{CIR}}])^2}.
$$

By applying the Esscher transform, defined as:
$$
\mathbb{P}_\lambda(d\omega) = \frac{e^{\lambda H_t^{CIR}(\omega)}}{\mathbb{E}_{p_0}[e^{\lambda H_t^{CIR}}]} \mathbb{P}(d\omega), 
$$
we observe that $\Lambda_t^{\prime}(\lambda) = \mathbb{E}_{\mathbb{P}_\lambda}[H_t^{CIR}]$ and $\Lambda_t^{\prime\prime}(\lambda) = \text{Var}_{\mathbb{P}_\lambda}(H_t^{CIR}) > 0$, since $\sigma > 0$ implies that $H_t^{CIR}$ is a non-degenerate random variable.

Given that $G^{\prime\prime}(\lambda) = \Lambda_t^{\prime\prime}(\lambda) > 0$, it follows that:
$$
f^{\prime\prime}(\lambda) = e^{G(\lambda)} [G^{\prime\prime}(\lambda) + (G^\prime(\lambda))^2] > 0. 
$$
Thus, $f$ is strictly convex. To prove the existence of a unique minimum in the interior of the interval, we examine the boundary behavior of $G^{\prime}(\lambda) = \Lambda_t^\prime(\lambda) - M$:

\begin{itemize}
    \item Using the Dominated Convergence Theorem again, $\lim_{\lambda \downarrow 0} \Lambda_t^\prime(\lambda) = \mathbb{E}_{p_0}[H_t^{CIR}]$. Since $M > \mathbb{E}_{p_0}[H_t^{CIR}]$, it follows that $\lim_{\lambda \downarrow 0} G^{\prime}(\lambda) < 0$.
    \item Note that $A(t, \lambda)$ and $B(t, \lambda)$ converge to $\infty$ as $\lambda \uparrow \lambda_c$; therefore, $\lim_{\lambda \uparrow \lambda_c} \Lambda_t^\prime(\lambda) = \infty$, which implies $\lim_{\lambda \uparrow \lambda_c} G^\prime(\lambda) > 0$. 
\end{itemize}
By the Intermediate Value Theorem, there exists a unique $\lambda^*(M) \in (0, \lambda_c)$ such that $G^\prime(\lambda^*) = 0$, confirming that $f$ reaches a global minimum.

(ii) The optimal point $\lambda^*(M)$ is implicitly defined by the condition $\Lambda_t'(\lambda^*) = M$. Applying the Implicit Function Theorem, we obtain:
\begin{equation*}
    \frac{d\lambda^*}{dM} = \frac{1}{\Lambda_t''(\lambda^*)} = \frac{1}{\text{Var}_{\mathbb{P}_{\lambda^*}}(H_t^{CIR})} > 0.
\end{equation*}
This shows that $\lambda^*(M)$ is strictly increasing in $M$. Furthermore, since $\Lambda_t'(\lambda)$ is a continuous, strictly increasing mapping from $(0, \lambda_c)$ to $(\mathbb{E}_{p_0}[H_t^{CIR}], \infty)$, the condition $M \to \infty$ forces $\lambda^*(M)$ to approach $\lambda_c$ to satisfy the equality $\Lambda_t'(\lambda^*) = M$.
\end{proof}

\section{Numerical studies}

\subsection{The Jacobi Case: Simulations and Discussion}

To assess the practical dynamics of the model, we conduct numerical simulations using the Euler--Maruyama discretization scheme. We generate an ensemble of $N_{\text{sim}} = 50$ trajectories over a time horizon $T=15.0$ with a time step of $\Delta t = 5 \times 10^{-3}$. For these simulations, the parameter set is fixed at $(\theta, \mu, \sigma) = (2.0, 0.4, 0.3)$ and the upper bound is set to $a = 1.0$, with initial conditions $p_0 = 0.5$ and $S_0 = 0.99$.

\paragraph{Baseline Scenario: Non-intervention ($\varphi \equiv 1$).}
In this case, we examine the system's natural evolution without external control. As illustrated in Figure \ref{fig:jacobi_baseline}, the transmission intensity $P_t^J$ exhibits stochastic fluctuations strictly confined within the interval $(0, a)$. Consistent with Theorem \ref{no-intervention}, the integrated intensity $H_t$ grows linearly on average, driving the epidemic to eventual saturation. The bottom panel of Figure \ref{fig:jacobi_baseline} confirms that the infected fraction $1-S_t$ converges to unity for all sample paths, indicating that the entire population becomes infected in the absence of mitigation measures.

\begin{figure}[htbp]
    \centering
    \includegraphics[width=0.6\textwidth]{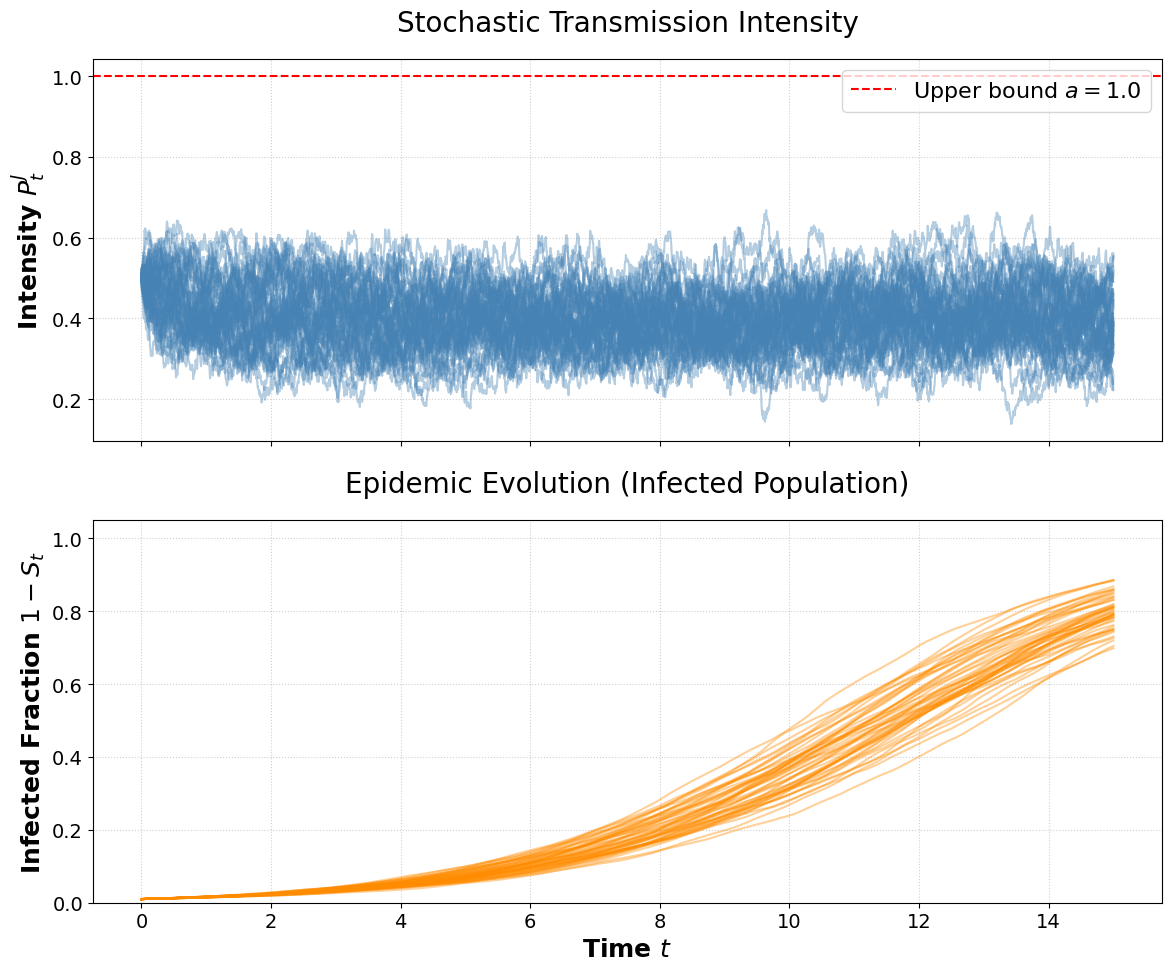} 
    \caption{Stochastic trajectories of the Jacobi transmission intensity $P^J_t$ (top) and the corresponding infected fraction $1-S_t$ (bottom) for the baseline case $\varphi \equiv 1$. The trajectories illustrate the inevitable saturation of the epidemic ($1-S_t \to 1$) as the cumulative intensity $H_t$ diverges.}
    \label{fig:jacobi_baseline}
\end{figure}
\paragraph{Scenario with Exponential Intervention.}
We now analyze the impact of a public health intervention modeled by the function $\varphi(t) = e^{-0.2t}$. In this scenario, the effective transmission rate $\beta_t = \varphi(t) P_t^J$ is gradually suppressed, representing a continuous strengthening of control measures. Figure \ref{fig:jacobi_exp} provides numerical evidence supporting Theorem \ref{intervention}. Although the underlying stochastic process $P_t^J$ remains active, the multiplicative effect of the decay function $\varphi(t)$ ensures that the integrated intensity $H_t$ converges to a finite random variable $H_\infty$ almost surely. Consequently, the epidemic growth is halted, and a significant fraction of the population remains susceptible in the long term ($S_\infty > 0$), as shown in the bottom panel.

\begin{figure}[H]
    \centering
    \includegraphics[width=0.6\textwidth]{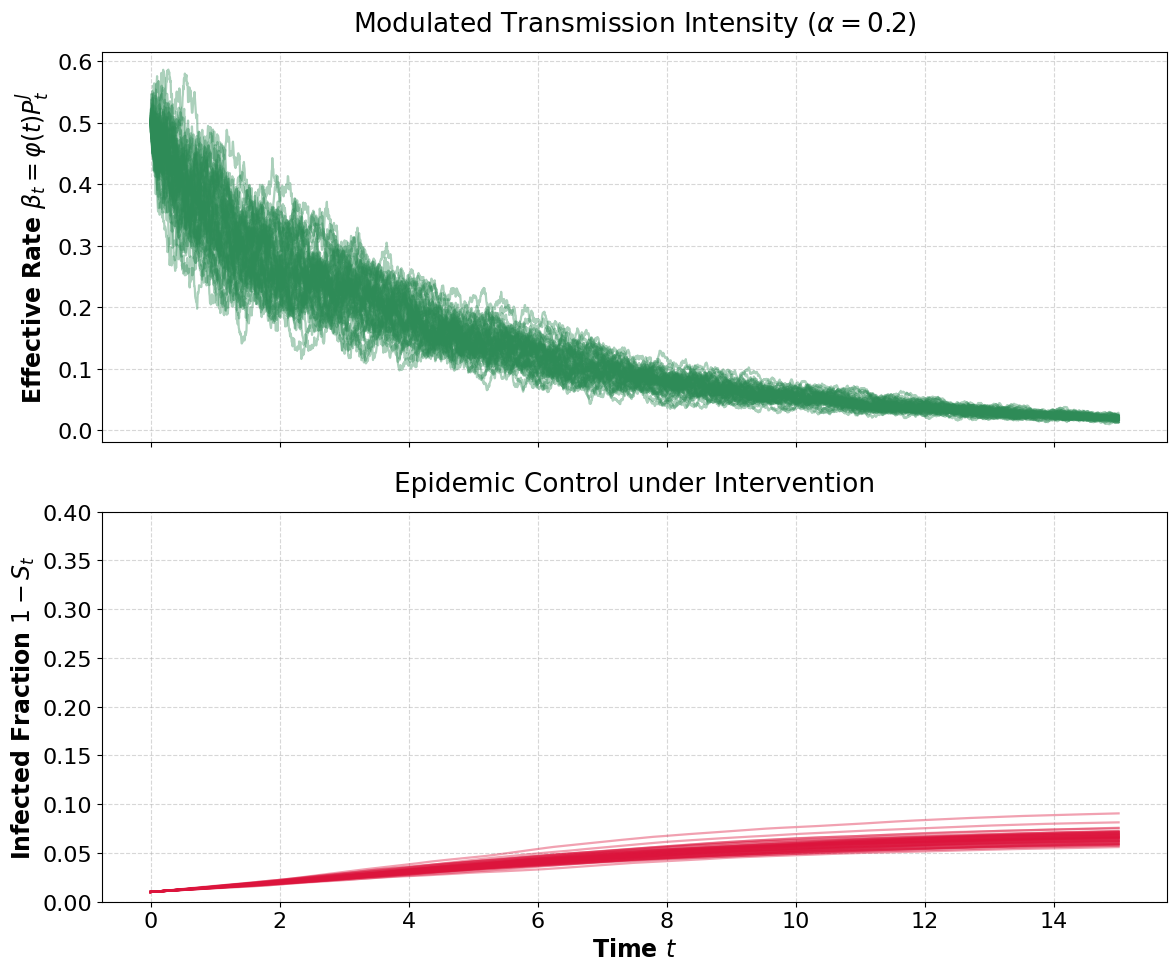}
    \caption{System dynamics under exponential intervention $\varphi(t) = e^{-0.2t}$. The top panel shows the effective transmission rate $\beta_t = \varphi(t)P_t^J$ decaying towards zero, while the bottom panel illustrates how the infected fraction $1-S_t$ stabilizes at a value strictly less than one, preserving a susceptible fraction $S_\infty > 0$.}
    \label{fig:jacobi_exp}
\end{figure}

\subsection{The CIR Case: Simulations and Discussion}

Numerical simulations for the CIR framework were performed using the same discretization settings ($T=15$, $\Delta t=5\times 10^{-3}$) and initial conditions ($p_0=0.5$, $S_0=0.99$) as in the Jacobi case. We generated $N_{\text{sim}} = 50$ independent trajectories. The structural parameters for the CIR process were set to $(\kappa, \eta, \sigma) = (2.0, 0.4, 0.3)$, satisfying the Feller condition ($2\kappa\eta > \sigma^2$) to ensure the transmission rate remains strictly positive.

\paragraph{Baseline Scenario: Non-intervention ($\varphi \equiv 1$).}
In the absence of control measures, the transmission rate is driven solely by the CIR process ($\beta_t = P_t^C$). Under this regime, the integrated intensity $H_t$ diverges to infinity almost surely as $t \to \infty$. Consequently, the entire population eventually becomes infected, as illustrated in Figure \ref{fig:cir_baseline}, where the infected fraction $1-S_t$ asymptotically approaches unity for all simulated paths.

\begin{figure}
    \centering
    \includegraphics[width=0.6\textwidth]{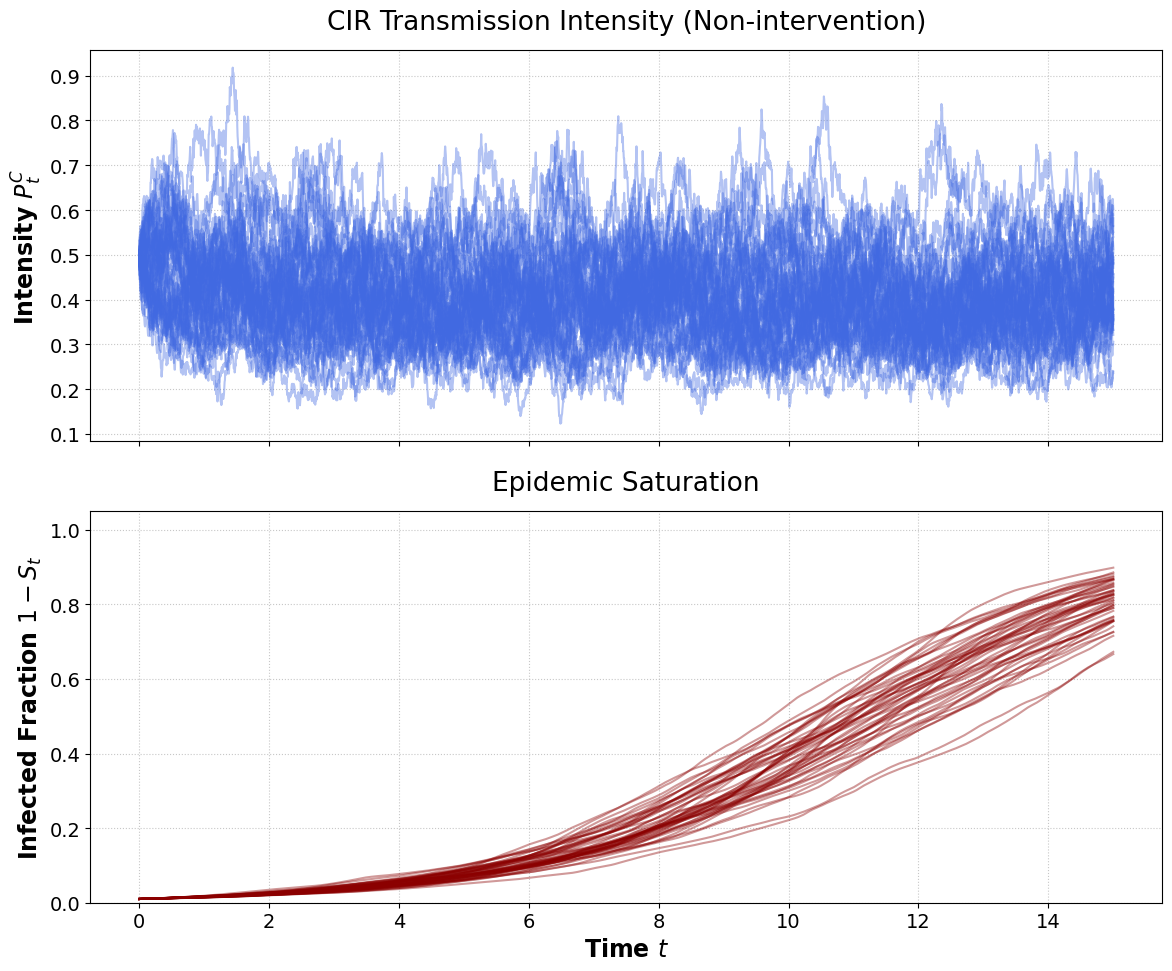} 
    \caption{Temporal evolution of the CIR transmission rate $P_t^{C}$ (top) and the resulting infected fraction $1-S_t$ (bottom) for $\varphi \equiv 1$. As $t$ increases, the cumulative pressure $H_t$ leads the system to a state where $1-S_t \approx 1$ almost surely.}
    \label{fig:cir_baseline}
\end{figure}

Numerical simulations were performed using the discretization parameters ($T=15, \Delta t=5\times 10^{-3}$) and initial conditions ($p_0=0.5, S_0=0.99$). The parameters for the CIR process were set to $(\kappa, \eta, \sigma) = (2, 0.4, 0.3)$ and we used $N_{\text{sim}} = 50$ trajectories.

\paragraph{Scenario with exponential Intervention.}
We now consider simulations with the exponential intervention $\varphi(t)=e^{-0.2t}$. As Figure \ref{fig:cir_exp} illustrates, the transmission rate $\beta_t = \varphi(t)P_t^C$ decays rapidly enough to ensure the integrated intensity process $H_t$ remains a.s. finite. This creates a ``continuous curbing effect'', preventing the eventual infection of the entire population.

\begin{figure}
    \centering
    \includegraphics[width=0.6\textwidth]{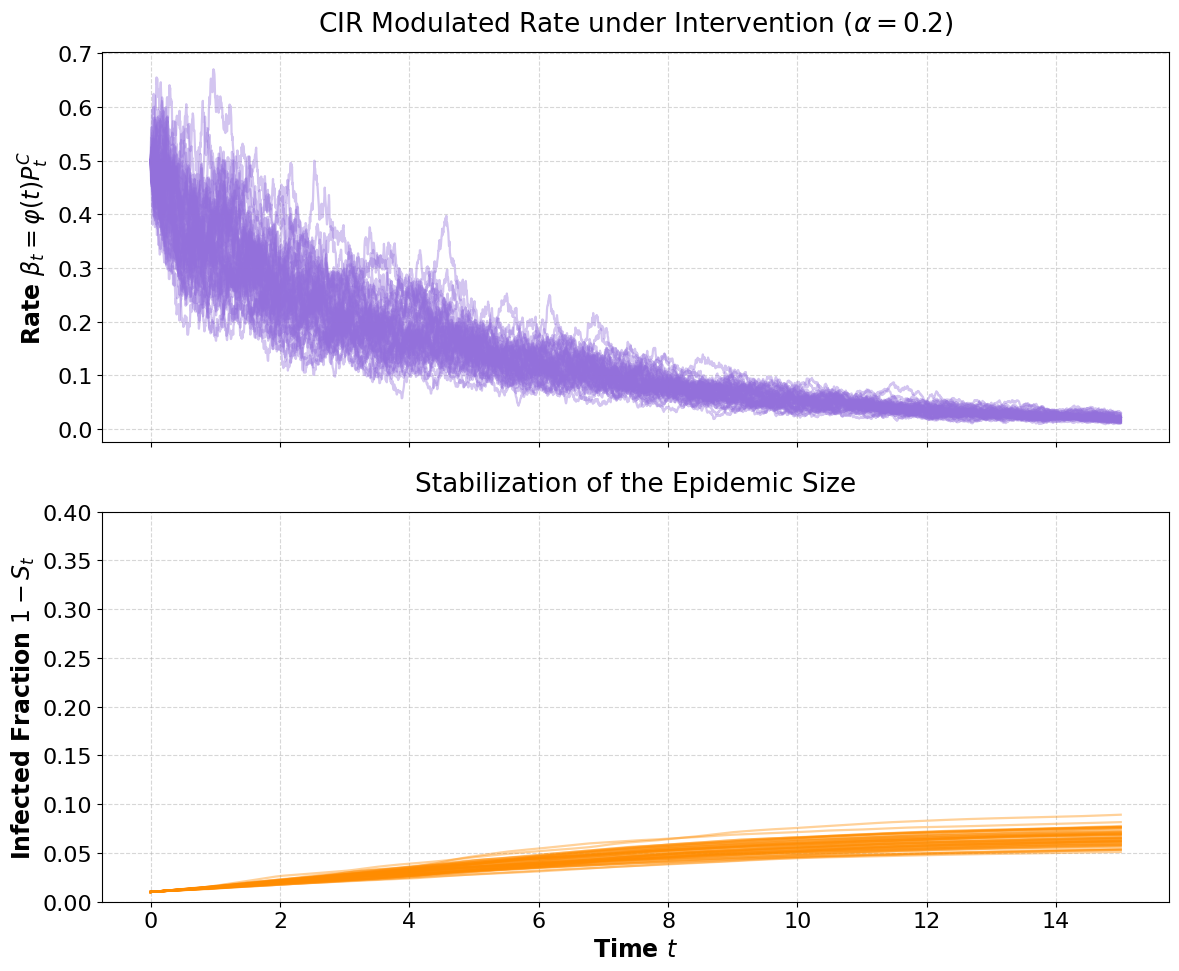}
    \caption{The epidemic dynamics under the intervention $\varphi(t) = e^{-0.2t}$. The top panel displays the effect of the intervention on the transmission rate, while the bottom panel shows the resulting impact on the fraction of infected individuals.}
    \label{fig:cir_exp}
\end{figure}

\textit{Remark.} The analytical tractability of the integrated intensity $(H_t^{CIR}; t \ge 0)$ under the CIR process in the non-intervention scenario provides significant advantages from a risk analysis perspective. As previously shown, the affine structure of the process yields a closed-form expression for the Laplace transform of $H_t^{CIR}$, which facilitates the calculation of its moments.

\subsection{Simulation of Threshold Times Without Intervention Under the CIR Process}

To numerically illustrate the random time to reach infection thresholds in the non-intervention CIR scenario, we set the model parameters to $\kappa=2$, $\eta=0.4$, $p_0=0.5$, and $t=1$. Under these settings, the expected integrated intensity is $\mathbb{E}_{p_0}[H_1^{CIR}] \approx 0.443$. To better explore the sensitivity of the optimal parameter and the corresponding risk bounds, we consider thresholds $M \in \{0.3, 0.5, 0.55, 0.6, 0.65, 0.7\}$. As illustrated in Figure~\ref{fig:chernoff_regimes}, when $M=0.3$, the bound is trivial (it equals 1). However, for values of $M$ above the mean, $\lambda^*(M)$ rapidly shifts towards the singularity $\lambda_c$. Crucially, the minimum value of $f(\lambda)$, which represents the upper bound for the probability $\mathbb{P}(\tau \le 1)$, decreases significantly as $M$ increases. For $M=0.5$, the probability bound is approximately $0.88$, while for a more severe threshold like $M=0.7$, it drops to $0.45$. This demonstrates that while the optimal parameter is constrained by the singularity, the risk of reaching high infection levels before time $t=1$ decreases as expected.

\begin{figure}[H]
    \centering
    \includegraphics[width=0.85\textwidth]{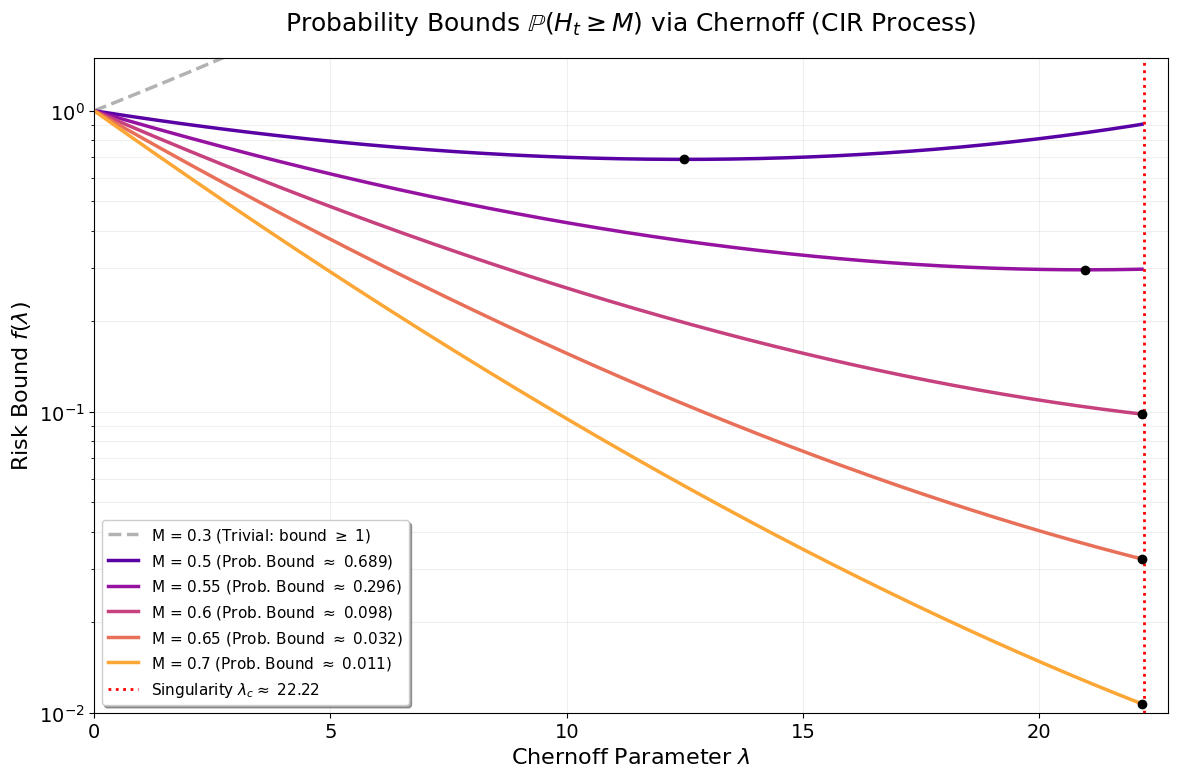}
    \caption{Sensitivity of the risk bound $f(\lambda)$ and corresponding probability bounds across different regimes of $M$. The legend displays the approximate value of the tightest Chernoff bound $f(\lambda^*)$, which acts as an upper bound for the probability of reaching the infection threshold before time $t=1$.}
    \label{fig:chernoff_regimes}
\end{figure}
\section{Numerical Comparison of the Jacobi and CIR Processes for $P_t$}
\label{sec:numerico}

\subsection{Experimental Setup and Calibration}

The core of our methodology relies on isolating the effect of the process support. Specifically, we aim to contrast the dynamics of the bounded Jacobi process, $P_t^J$, with those of the unbounded CIR process, $P_t^C$, ensuring that no other statistical factors bias the comparison. To achieve this, we calibrate both models to share the same first two stationary moments, allowing us to observe how the difference in their supports affects the epidemic's progression. 

We first align the drift components by setting an identical mean-reversion speed ($\theta_J = \kappa = 0.5$) and a shared long-term average ($\mu = \eta = 0.3$) for both models. For the Jacobi process, we fixed an upper bound of $a = 0.5$ to represent a physical or biological limit on transmission intensity.

The main challenge in this setup lies in the calibration of the diffusion coefficients. Given that the Jacobi process is physically constrained by its boundaries, its stationary variance is defined by the expression:
\begin{equation}\label{eq:var_jacobi}
\text{Var}[P_\infty^J] = \frac{\sigma_J^2 \mu (a - \mu)}{2\theta_J + a\sigma_J^2}.
\end{equation}
Our procedure consists of selecting a value for $\sigma_J$, calculating the resulting variance, and then solving for the specific $\sigma_{CIR}$ that yields the exact same level of dispersion in the CIR process (where $\text{Var}[P_\infty^C] = \sigma_{CIR}^2 \eta / (2\kappa)$). This ensures that we are strictly controlling for the amount of noise in the system. Consequently, any discrepancy observed in the epidemic dynamics can be attributed solely to the boundary effects rather than differences in the first two moments. To guarantee the robustness of these findings, particularly for rare tail events, we initialize the simulation at $s_0 = 0.99$ and execute $N=100,000$ Monte Carlo paths.
\subsection{Scenario 1: Dynamics under Unchecked Transmission ($\varphi \equiv 1$)}

In the absence of intervention, both processes eventually lead to the infection of the entire population ($S_t \to 0$ as $t \to \infty$). The objective is to identify how the stochastic drivers differ before this terminal state is reached.

\paragraph{Low Volatility Regime.}
By setting $\sigma_J = 0.2$, we obtain a corresponding $\sigma_{CIR} \approx 0.089$ via the moment-matching calibration. We observe the system at a time horizon of $T=15$. As displayed in Figure \ref{fig:scenario1_1}, the resulting probability density functions for both models are virtually identical. In this low-volatility regime, the process rarely reaches the upper bound $a=0.5$; consequently, the Jacobi process behaves like an unbounded process near its mean, making the structural difference between the two models negligible.

\begin{figure}
    \centering
    \includegraphics[width=0.85\textwidth]{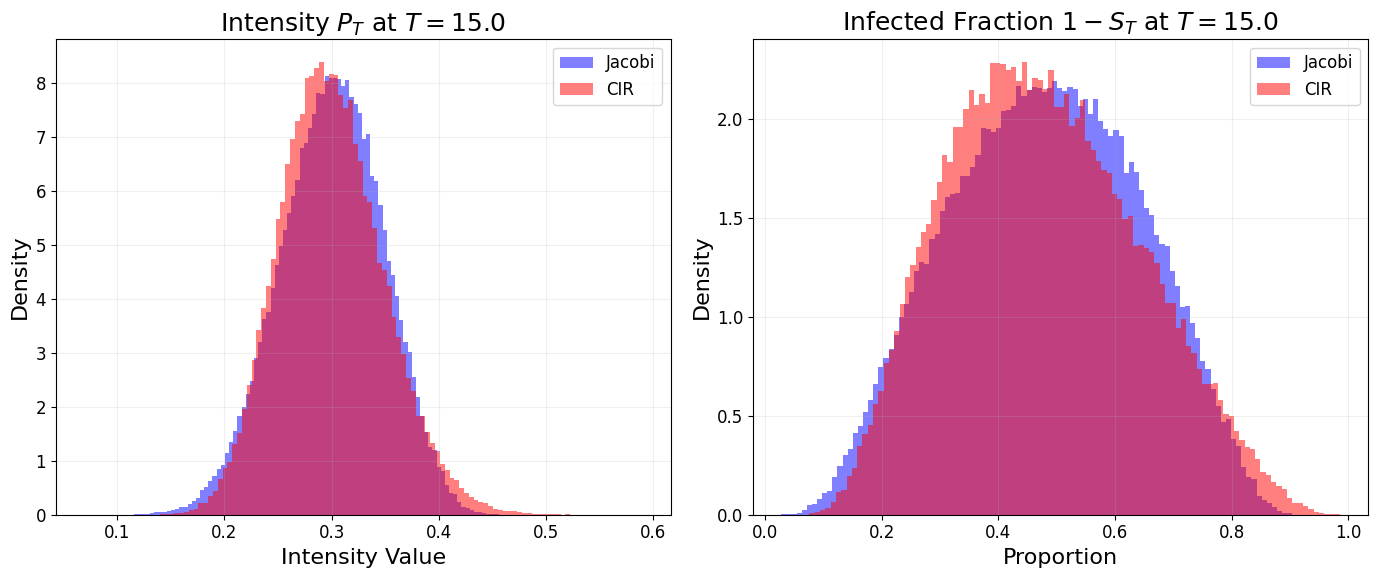}
    \caption{Comparison under Low Volatility ($\sigma_J = 0.2, \sigma_{CIR} \approx 0.089$). The left panel shows the infection intensity $P_T$, and the right panel shows the epidemic size $1-S_T$. The distributions overlap almost perfectly, demonstrating model equivalence when fluctuations remain far from the boundaries.}
    \label{fig:scenario1_1}
\end{figure}

\paragraph{High Volatility and Cumulative Pressure ($T=100$).} 
By setting the volatility to $\sigma_J = 1.5$ (with a corresponding $\sigma_{CIR} \approx 0.387$) over an extended horizon of $T=100$, we observe that the final epidemic size, $1-S_T$, saturates at $1.0$ for both models. However, the cumulative transmission intensity, $H_T = \int_0^T \beta_s \, ds$, highlights the fundamental differences between the two frameworks, as illustrated in Figure \ref{fig:scenario2}. The CIR process exhibits a heavy right tail because it admits extreme transmission intensities that are structurally non-realizable within the bounded Jacobi framework. This suggests that this tail behavior, rather than the mean dynamics, governs the elevated risk bounds in unbounded models.

\begin{figure}
    \centering
    \includegraphics[width=0.6\textwidth]{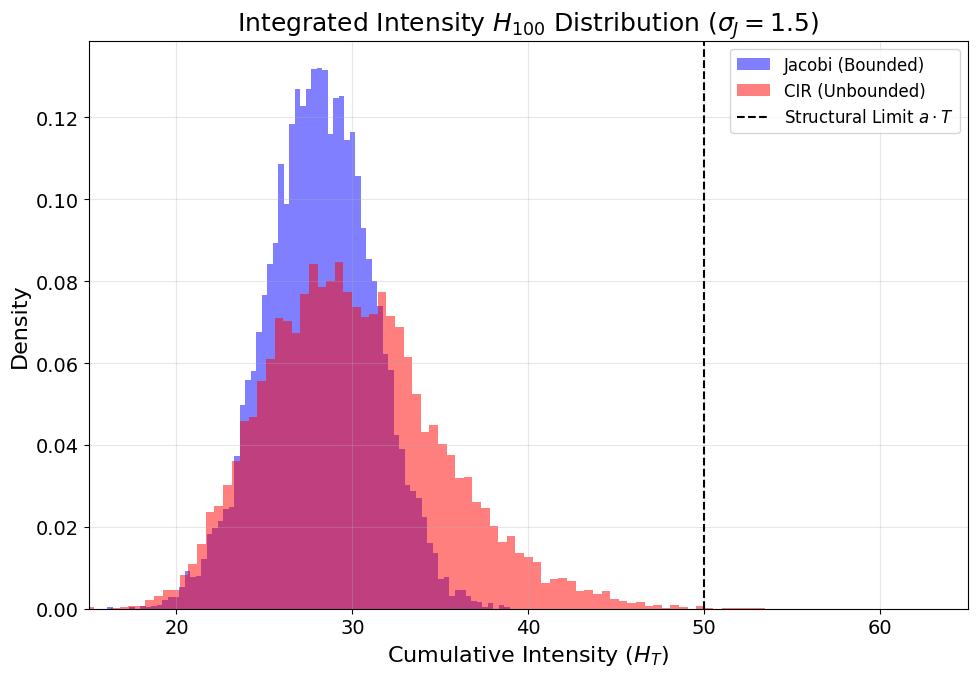}
    \caption{Comparison of the cumulative transmission intensity $H_T$ at $T=100$. While the epidemic size has saturated, the CIR model (red) assigns significant weight to extreme cumulative intensities that lie far beyond the structural limits of the Jacobi model (blue).}
    \label{fig:scenario2}
\end{figure}
\paragraph{Early Stage Divergence ($T=10$).}
When observing the process up to $t=10$, the disparity in the behavior of $H_t$ becomes apparent in the fraction of infected individuals. Given that the CIR process is unbounded, it admits significantly larger excursions than the Jacobi process, thereby accelerating the infection dynamics in extreme cases. As displayed in Figure~\ref{fig:scenario3}, the CIR process generates a heavier right tail in the distribution of the fraction of the population infected ($1-S_T$) during the early stages of the epidemic. This confirms that even if the stationary means are identical, the lack of a structural ceiling in the CIR framework allows for more rapid outbreaks.

\begin{figure}[H]
    \centering
    \includegraphics[width=0.6\textwidth]{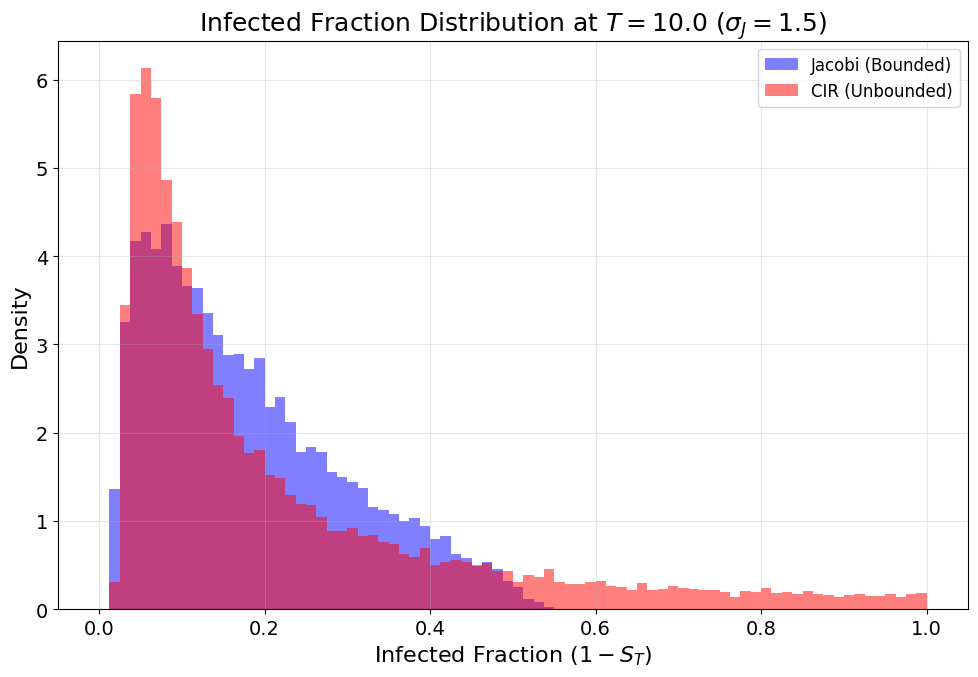} 
    \caption{Distribution of the infected fraction $1-S_T$ at the early stage $T=10$ under high volatility ($\sigma_J = 1.5, \sigma_{CIR} \approx 0.387$). The comparison shows that the CIR model (red) produces a heavier right tail, indicating a higher probability of rapid infection spread compared to the bounded Jacobi model (blue).}
    \label{fig:scenario3}
\end{figure}

\subsection{Scenario 2: Impact of Mitigation Measures ($\varphi(s)=e^{-\alpha s}$)}

Next, we evaluate the models under a stringent intervention regime, $\varphi(s)=e^{-0.2s}$. 

\paragraph{Tail Risk and Model Comparison.}
In the presence of high volatility ($\sigma_J=1.5$), the discrepancies between the two models become increasingly pronounced. Figure~\ref{fig:scenario_intervention} illustrates the final epidemic size, $1-S_\infty$, under this exponential intervention. This divergence is primarily due to the heavier right tail of the epidemic distribution under the CIR model. As summarized in Table~\ref{tab:intervention_metrics}, the extreme outcomes differ significantly: the 95th percentile (Value at Risk) for the Jacobi model is approximately $0.320$, whereas for the CIR model, it reaches $0.530$. These results show that, despite identical calibration of the mean and variance, the adoption of an unbounded support yields considerably more conservative estimates of outbreak severity under high volatility.

\begin{table}[htbp]
\centering
\caption{Comparison of Risk Metrics for the Final Epidemic Size ($1-S_\infty$) under intervention $\varphi(s) = e^{-0.2s}$.}
\label{tab:intervention_metrics}
\begin{tabular}{lcc}
\toprule
Metric & Jacobi (Bounded) & CIR (Unbounded) \\
\midrule
Mean Epidemic Size & 0.185 & 0.192 \\
95th Percentile (VaR) & \textbf{0.320} & \textbf{0.530} \\
Maximum Observed Size & 0.460 & 0.985 \\
\bottomrule
\end{tabular}
\end{table}

\begin{figure}[htbp]
    \centering
    \includegraphics[width=0.6\textwidth]{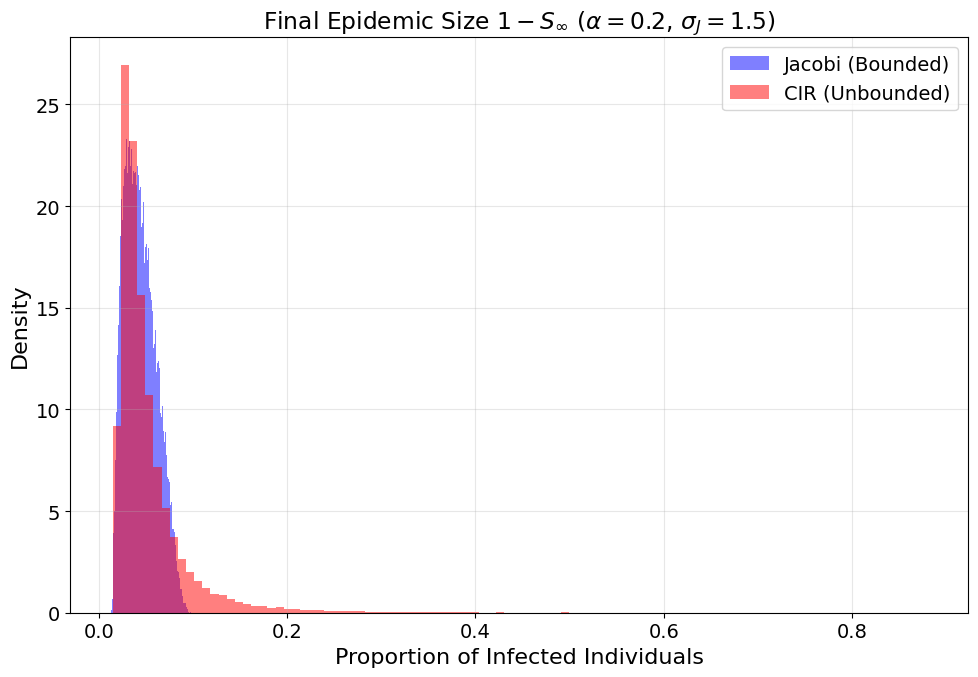}
    \caption{Final Epidemic Size under exponential intervention $\varphi(s) = e^{-0.2s}$. The bounded nature of the Jacobi process (blue) prevents the epidemic from reaching extreme sizes, while the CIR model (red) assigns significant probability to severe outbreaks despite the intervention.}
    \label{fig:scenario_intervention}
\end{figure}
\paragraph{Interpretation of Tail Divergence.}
These results are consistent with the theoretical properties of the chosen models. As expected, the unbounded nature of the CIR process allows for higher excursions in the transmission rates compared to the Jacobi model, which is strictly constrained by its upper bound~$a$. Even when both processes are calibrated to the same stationary moments, the lack of a ceiling in the CIR framework naturally leads to a heavier right tail in the distribution of the final epidemic size. 
\section{Discussion and Conclusions}
\label{sec:conclusion}

In this work, we compared the Jacobi and CIR processes as stochastic drivers for epidemic dynamics under intervention. Our analysis primarily centered on the CIR framework due to its mathematical tractability and its capacity to capture extreme fluctuations through heavier tails. We demonstrated that even when both models are statistically calibrated to share the same stationary mean and variance, the choice of support (bounded vs. unbounded) leads to significantly different risk assessments, particularly in high-volatility regimes.

Despite these insights, several statistical open problems remain to be addressed in future work. For instance, given an empirical epidemic dataset, can we develop robust estimation techniques to recover both the underlying process parameters and the deterministic intervention function? Furthermore, characterizing the specific inferential challenges posed by each model and deriving closed-form expressions for particular classes of intervention functions remain compelling avenues for further research.

\bibliographystyle{plainnat}
\bibliography{referencias2}

\end{document}